\documentclass[12pt]{amsart} 

\usepackage{amsmath}
\usepackage{amssymb}
\usepackage{mathtools}
\usepackage[utf8]{inputenc}
\usepackage{tikz-cd}
\usepackage{hyperref}

\newtheorem{theorem}{Theorem}[section]
\newtheorem{corollary}[theorem]{Corollary}
\newtheorem{proposition}[theorem]{Proposition}
\newtheorem{lemma}[theorem]{Lemma}

\theoremstyle{definition}
\newtheorem{definition}[theorem]{Definition}
\newtheorem{example}[theorem]{Example}

\theoremstyle{remark}

\DeclareMathOperator{\Pers}{\mathrm {Mor(Pers)}}
\DeclareMathOperator{\match}{\mathrm {Match}}

\DeclareMathOperator{\id}{\mathrm {id}}

\DeclareMathOperator{\im}{\mathrm {im}}

\DeclareMathOperator{\coker}{\mathrm {coker}}

\newcommand{\xto}{\xrightarrow}

\newcommand{\cech}{\check{\mathcal C}}
\newcommand{\rstar}{[0,\infty]}
\newcommand{\rnonneg}{[0,\infty)}
\newcommand{\upper}{E}

\newcommand{\CC}{\mathcal C}
\newcommand{\DD}{\mathcal D}

\begin{document}

\title{Sparse Filtered Nerves}
\author{Nello Blaser}
\author{Morten Brun}

\begin{abstract}
  Given a point cloud \(P\) in Euclidean space and a positive
  parameter 
  \(t\) we 
  can consider the \(t\)-neighborhood \(P^{t}\) of \(P\)
  consisting 
  of points at distance less than \(t\) to \(P\). Homology of
  \(P^{t}\) gives information about components, holes, voids
  etc. in \(P^{t}\). The idea
  of persistent homology is that it may
  happen that we are interested in some of holes in the spaces \(P^t\)
  that are not detected simultaneously in homology for a single value of \(t\), but
  where each of these holes is detected for 
  \(t\) in a wide range.

  When the dimension of the ambient Euclidean space is small,
  persistent homology is efficiently computed by the
  \(\alpha\)-complex \cite{alpha}. For dimension bigger than three
  this becomes resource consuming. Don Sheehy discovered
  \cite{Sheehy2013} that 
  there exists a filtered simplicial complex whose size depends linearly
  on the 
  cardinality of \(P\) and whose persistent homology is an approximation
  of the persistent homology of the filtered topological space
  \(\{P^{t}\}_{t \ge 0}\). In this
  paper we pursue Sheehy's sparsification approach
  and give a more general approach to sparsification of
  filtered simplicial complexes computing the homology of filtered
  spaces of the form \(\{P^{t}\}_{t \ge 0}\) and more generally to
  sparsification of 
  filtered Dowker nerves. To our best knowledge, this is the first approach to 
  sparsification of general Dowker nerves.
\end{abstract}

\maketitle

\section{Introduction}
\label{sec:intro}

In data analysis, we often want to quantify an underlying shape of data. For example, in cluster analysis the hypothesis is
that data is concentrated in certain regions and in linear regression the
hypothesis is that data is concentrated along a line. The main purpose
of 
topological data analysis is to discover and quantify more complicated
shapes like circles and spheres. 
Persistent homology is the preferred tool for this.

For example, 3D-scanning gives a sample of points on the surface of a
solid object in three dimensional space. The homology of this
surface contains information about the underlying solid object. Under
certain conditions persistent homology of the finite point sample
allows us to infer the homology of the surface we sample from \cite{Niyogi2008}. 

There are several versions of 
persistent homology of a point sample \(P\)
in Euclidean space \(E\).
The filtered \v Cech complex gives the homology \(H^{t} =
H_*(P^{t})\) of the 
spaces 
\[P^{t} = \{x \in E \, \mid \, d(x,P) < t\}\]
together with the homomorphisms \(H^{t} \to H^{t'}\) for
\(t \le t'\). In \cite{Niyogi2008} it is shown
that under 
favourable 
circumstances, if the point sample \(P\) is drawn from a compact
submanifold \(M\) of \(E\), then the homology of \(M\) can be inferred
from the \v Cech homology of \(P\).
Another version of persistent homology is based on the filtered Rips
complex of \(P\). This filtered simplicial complex is constructed from
the pairwise distances
between points in the sample \(P\). Given \(t > 0\), the
abstract simplicial complex \(R_{t}(P)\) is the
clique complex of the (simple undirected) graph with the set \(P\) as
nodes and with edges given by pairs of points in \(P\) of distance less than \(t\).
That is, given \(t > 0\), a subset \(\sigma\) of \(P\) is
contained in \(R_{t}(P)\) if and only the set of pairwise distances
between points in \(\sigma\) is strictly bounded above by \(t\).
Latschev \cite{Latschev2001} has shown that under certain
conditions, the homology of \(M\) can be inferred from the persistent
homology of this filtered simplicial complex.

Unfortunately, the sizes of both the Rips and of the \v Cech complex
explode when the number of points in the sample \(P\) grows.
Traditionally this is addressed in two ways. The \(\alpha\) complex has
the same persistent homology as the \v Cech complex and if the ambient Euclidean
space is of dimension at most three, then it is so small
that it is practical to compute even when \(P\) consists of millions of
points. On the other hand, discrete Morse theory allows us to replace the
Rips complex by a much smaller complex with the same persistent
homology, so that it is practical to compute persistent Rips homology
for tens of thousands of points \cite{10.1007/978-3-319-04099-8_7}. 
Sheehy's approximations to \v Cech and Rips complexes
\cite{Sheehy2013, SRGeom} can be used to  
push the limits on the number of data points. The sizes of these
approximations grow linearly in the
number of sample points. However it depends on constants that grow
exponentially in the dimension of the ambient Euclidean space.

In this work we follow the approach of Sheehy and Cavanna, Jahanseir
and Sheehy \cite{Sheehy2013, SRGeom}.
We modify their approximations so that they can be applied to
the Dowker nerve of arbitrary dissimilarities, that is,
functions of the form
\begin{displaymath}
  \Lambda \colon L \times W \to \rstar.
\end{displaymath}
The Dowker nerve of \(\Lambda\) is a filtered simplicial complex \(N(\Lambda)\) with vertex
set \(L\). Given \(t > 0\) and \(w \in W\), every finite subset of the set
\begin{displaymath}
  \{l \in L \, \mid \, \Lambda(l,w) < t\}
\end{displaymath}
is a face of the simplicial complex \(N_t(\Lambda)\), and
\(N_t(\Lambda)\) is the smallest simplicial complex containing
these faces.

Note that if \(L\) is a subset of a metric space \(W\) and \(\Lambda
\colon L \times W \to \rstar\)
is given by the metric, then
\(N_t(\Lambda)\) can be described as
the {\em ambient
  \v{C}ech complex \(\cech_t(L, W)\)}. That is, it is the nerve of the set of
\(t\)-balls in \(W\) centred at points in \(L\) considered as a cover
of the union of such balls.
Considering \(L\) as a metric space with
the induced metric, we call \(\cech_t(L,L)\) the {\em intrinsic
  \v{C}ech complex of \(L\)}.  From the perspective of computer implementation
these relative \v{C}ech complexes and their Dowker counterparts have
the defect that they grow rapidly when the size of \(L\) increases. In
order to mitigate this Sheehy, Botnan--Spreemann and
Cavanna--Jahanseir--Sheehy proposed sparse approximations to
\(\cech_t(L, W)\) in the situation where \(W = \mathbb{R}^d\) equipped
with a convex metric and \(L\) is a finite subset of \(W\)
\cite{Sheehy2013, botnan15approximating, SRGeom}.  Inspired by their
work, in \cite{SparseDowker}, we constructed sparsifications of nerves
of dissimilarities satisfying the triangle inequality. In this
paper we construct sparsifications of arbitrary
dissimilarities, that is, arbitrary functions of the form
\begin{displaymath}
  \Lambda \colon L \times W \to \rstar.
\end{displaymath}
In the situation where \(L\) and \(W\) are finite
and all the values \(\Lambda(l,w)\) for \((l,w) \in L \times W\) are
stored in memory these sparsifications can be implemented on a
computer in a direct way.

  Our sparsification can be considered as  
  consisting of two parts.  In the first part we replace a metric by a
  Dowker dissimilarity whose nerve is an approximation of the
  \v{C}ech nerve of the original metric.  In the second part we replace
  the nerve of a Dowker dissimilarity by a smaller homotopy equivalent
  filtered simplicial complex. This filtered simplicial complex is the
  smallest member of a class of sparsifications including the ones in
  \cite{SRGeom} and \cite{SparseDowker}. These two parts are
  intimately related and it is not evident how to combine them to
  obtain a smallest possible sparsification.

The intrinsic and the ambient \v{C}ech complexes are related by the
inclusions
\begin{displaymath}
  \cech_t(L,L) \subseteq \cech_t(L,W) \subseteq \cech_{2t}(L,L),
\end{displaymath}
so their corresponding persistent homologies are multiplicatively
\(2\)-interleaved. The ambient \v{C}ech complex has homotopy type
given by the sublevel filtration for the function
\(f \colon W \to [0,\infty)\) whose value on a point in \(W\) is its
distance to \(L\). If \(L\) is contained in a metric subspace \(N\) of
\(W\), then the persistent homologies of \(\cech(L,W)\) and
\(\cech(N,W)\) are additively interleaved with interleaving factor
given by the Hausdorff distance \(d_H(L,N)\) between \(L\) and \(N\).
Moreover, the persistent homologies of \(\cech(L,L)\) and
\(\cech(N,N)\) are additively interleaved with the factor \(2d_H(L,N)\), that is,
two times the Hausdorff distance between \(L\) and \(N\). Thus the
persistent homologies of both \(\cech(L,L)\) and \(\cech(L,W)\) can be
considered as an approximations to the intrinsic \v{C}ech homology of
\(N\). In particular if \(W\) is a Riemannian manifold with distance
given by the geodesic metric, then the persistent homologies of
\(\cech(L,W)\) and \(\cech(W,W)\) are additively interleaved with
interleaving factor given by the Hausdorff distance \(d_H(L,W)\)
between \(L\) and \(W\).  At filtration values up to the convexity
radius of \(W\) the persistent homology \(\cech(W,W)\) is isomorphic
to the homology of \(W\). (See e.g. \cite[Section 6.5.3]{MR2002701})

All other sparsification strategies that we are aware of, explicitly
sparsify \v{C}ech or Rips complexes and leverage the metric structure
of the underlying space in order to sparsify. As already discussed,
Sheehy took advantage of the fact that the Euclidean metric is a
doubling metric \cite{Sheehy2013,SRGeom}. Choudhary and others
recently suggested an approximation based on discretizing Euclidean
space \cite{DBLP:journals/corr/abs-1812-04966}. Other approximations
work for Euclidean space \cite{botnan15approximating} or for Rips
\cite{simba} or weighted Rips complexes \cite{buchet16efficient}. We
have previously shown that a similar approach as Sheehy used could be
extended to Dowker dissimilarities that satisfy a triangle inequality
\cite{SparseDowker}. Here we present our first sparsification approach
applicable to general Dowker nerves.

This paper is organized as follows. Section~\ref{sec:preliminaries}
introduces the reader to the basic concepts used throughout the
remaining sections. 
In our previous paper \cite{SparseDowker}, we did not explain how
interleavings with respect to translation functions (see
Definition~\ref{translationfunction}) are related to matchings. Since
this is crucial to the interpretation of persistence diagrams of
sparse nerves we discuss this in Section~\ref{sec:matching}.  
In Section~\ref{sec:truncated}, we introduce 
truncation of Dowker nerves and give a direct argument 
showing that the truncated Dowker nerve is interleaved with the
Dowker nerve of the original Dowker dissimilarity. 
In Section~\ref{sec:dnerves}, 
we sparsify Dowker nerves in a way that preserves homotopy type. In
particular, persistent homology
does not change under sparsification. This sparsification is obtained
via a function \(R \colon L \to \rstar\) having certain
properties. 
We call functions that satisfy these properties {\em restriction
functions}. With the concept of restriction
functions at hand we display the smallest restriction
function relative to a parent function \(\varphi\), the {\em
  \((\Lambda, \varphi)\)-restriction}. In the paper \cite{appliedpaper} 
, we give a short description of
details behind our
python package for computation of persistent homology of sparsified
Dowker nerves.

\section{Preliminaries}
\label{sec:preliminaries}

\subsection{Filtratrations}
\label{sec:filtrations}
We consider the interval \(\rstar\) as a category with the underlying
set of the interval as objects and with a morphism \(s \to t\) if and
only if \(s \le t\).
\begin{definition}
  Let \(\CC\) be a category. The {\em category of filtered objects in
    \(\CC\)} is the category of functors from \(\rstar\) to
  \(C\). That is, a filtered object in \(\CC\) is a functor
  \(C \colon \rstar \to \CC\) and a morphism \(f \colon C \to C'\) of
  filtered objects in \(\CC\) is a natural transformation.
\end{definition}

Recall that a function \(\alpha \colon \rnonneg \to \rnonneg\) is {\em
  order preserving} if it is a functor, that is, if \(s
\le t\) implies 
\(\alpha(s) \le \alpha(t)\).
\begin{definition}\label{generalizedinverse}
  Let \(\beta \colon \rnonneg \to \rnonneg\) be an order preserving
  function with \(\lim_{t \to \infty}\beta(t) = \infty\). The {\em generalized
    inverse} of \(\beta\) is the order preserving function
  \begin{displaymath}
    \beta^{\leftarrow} \colon \rnonneg \to \rnonneg, \quad
    \beta^{\leftarrow}(t) = \inf \{ s \in \rnonneg \, \mid \, \beta(s) \ge t\}
  \end{displaymath}
  with the defining property
  \begin{displaymath}
    \beta^{\leftarrow} (t) \le s \text{ if and only if } t \le \beta(s).
  \end{displaymath}
\end{definition}
In categorical language, the generalized of \(\beta\) is its left
adjoint functor. 
\begin{definition}\label{translationfunction}
  A {\em translation function} is an order preserving function
  \(\alpha \colon \rnonneg \to \rnonneg\) with \(t \le \alpha(t)\) for
  every \(t \in \rnonneg\).
\end{definition}
In categorical language, a translation function is a functor under the
identity.

We will often allow ourselves to evaluate a translation \(\alpha\) at
\(\infty\) by letting \(\alpha(\infty) = \infty\). 
\begin{definition}
  Given a filtered object \(C \colon \rstar \to \CC\) and a
  translation function \(\alpha \colon \rnonneg \to \rnonneg\), the
  {\em pull-back of \(C\) along \(\alpha\)} is the filtered object
  \(\alpha^* C = C \circ \alpha\) with
  \((\alpha^* C)(t) = C({\alpha(t)})\). The {\em \(\alpha\)-unit} of
  \(C\) is the morphism \(\alpha_{*C} \colon C \to \alpha^*C\) with
  \[\alpha_{*C}(t) = C(t \to \alpha(t)) \colon C(t) \to (\alpha^*C)(t) =
  C(\alpha(t)).\] 
\end{definition}
\begin{definition}
  Let \(k\) be a field.  The category of {\em persistence modules over
    \(k\)} is the category of filtered objects in the category of
  vector spaces over \(k\).
\end{definition}
\begin{definition}
  Let \(k\) be a field and let \(\alpha \colon \rnonneg \to \rnonneg\)
  be a translation function.  A persistence module \(V\) over \(k\) is
  {\em \(\alpha\)-trivial} if the \(\alpha\)-unit of \(V\) is trivial,
  that is, if \(\alpha_{*V} = 0\).
\end{definition}

\subsection{Dissimilarities}

Our presentation of this preliminary material closely follows \cite{SparseDowker}.
\begin{definition}[Dowker \cite{MR0048030}]
  The {\em nerve} of a relation \(R \subseteq X \times Y\) is the
  simplicial complex
  \begin{displaymath}
    NR = \{ \text{ finite } \sigma \subseteq X \, \mid \, \exists
    \text{ \(y \in Y\) with \((x,y) \in R\) for all \(x \in
      \sigma\)}\}.  
  \end{displaymath}
\end{definition}

The following definition is inspired by the concept of networks as it
appears in \cite{2016arXiv160805432C}.
\begin{definition}\label{filtereddowkermorphism}
  A {\em dissimilarity} \(\Lambda\) consists of two sets \(L\)
  and \(W\) and a function \(\Lambda \colon L \times W \to \rstar\).
  Given \(t \in \rstar\), we let \(\Lambda_t \subseteq L \times W\) be
  the relation
  \begin{displaymath}
    \Lambda_t =
    \{(l,w) \in L \times W \, \mid \, \Lambda(l,w) < t\}.
  \end{displaymath}
\end{definition}

\begin{definition}
  Let \(\Lambda \colon L \times W \to \rstar\) be a
  dissimilarity.  The {\em Dowker Nerve} \(N \Lambda\) of \(\Lambda\)
  is the filtered simplicial complex with vertex set \(L\) and the nerve
  \(N \Lambda_t\) of the relation \(\Lambda_t\) in filtration degree
  \(t \in [0, \infty]\).
\end{definition}
\begin{definition}
  A {\em morphism} \(C \colon \Lambda \to \Lambda'\) of 
  dissimilarities \(\Lambda \colon L \times W \to [0,\infty]\) and
  \(\Lambda' \colon L' \times W' \to [0,\infty]\) consists of a
  relation
  \[C \subseteq L \times L'\] so that for every \(t \in [0,\infty]\)
  and for every \(\sigma \in N\Lambda_t\), the set
  \begin{displaymath}
    NC(\sigma) = \{ l' \in L' \, \mid \, (\sigma \times
    \{l'\}) \cap C 
    \text{ is non-empty} \}
  \end{displaymath}
  is non-empty and contained in \(N\Lambda'\).  If
  \(C \subseteq L \times L'\) and \(C' \subseteq L' \times L''\) are
  morphisms \(C \colon \Lambda \to \Lambda'\) and
  \(C' \colon \Lambda' \to \Lambda''\) of dissimilarities, then
  the composition
  \begin{displaymath}
    C'C \colon \Lambda \to \Lambda''
  \end{displaymath}
  is the subset of \(L \times L''\) defined by
  \begin{displaymath}
    C'C = \{ (l,l'') \, \mid \text{there exists } l' \in L' \text{
      with } (l,l') \in C \text{ and } (l',l'') \in C'\}.
  \end{displaymath}
  The identity morphism \(\Delta_L \colon N\Lambda \to N\Lambda\) is
  \begin{displaymath}
    \Delta_L = \{(l,l) \, \mid \, l \in L\} \subseteq L \times L.
  \end{displaymath}

\end{definition}
\begin{proposition}
  The Dowker nerve is functorial in the sense that it induces a
  functor \(N\) from the category of dissimilarities to the
  category of functors from \([0, \infty]\) to the category
  of topological spaces.
\end{proposition}
\begin{proof}
  Let \(C \subseteq L \times L'\) and \(C' \subseteq L' \times L''\)
  be morphisms \(C \colon \Lambda \to \Lambda'\) and
  \(C' \colon \Lambda' \to \Lambda''\) of 
  dissimilarities. Given \(t \in \rstar\), the functions
  \begin{displaymath}
    NC \colon N\Lambda_t \to N\Lambda'_t 
    \quad \text{ and } \quad
    NC' \colon N\Lambda'_t \to N\Lambda''_t 
  \end{displaymath}
  are order preserving. Thus, they induce morphisms of geometric
  realizations of barycentric subdivisions. The identity morphism
  \(\Delta_L \colon \Lambda \to \Lambda\) induces the identity
  function \(N\Delta_L \colon N\Lambda \to N\Lambda\). In order to
  finish the proof we show that \(NC'(NC(\sigma)) = N(C'C)(\sigma)\)
  for every \(\sigma \in N\Lambda_t\). If \(l'' \in NC'(NC(\sigma))\),
  then there exists \(l' \in NC(\sigma)\) so that \((l',l'') \in
  C'\). Since \(l' \in NC(\sigma)\) there exists \(l \in \sigma\) so
  that \((l,l') \in C\). We conclude that \((l,l'') \in C'C\) and thus
  \(l'' \in N(C'C)(\sigma)\). Conversely, if
  \(l'' \in N(C'C)(\sigma)\), then there exists \(l \in \sigma\) so
  that \((l,l'') \in C'C\). By definition of \(C'C\) this means that
  there exists \(l' \in L'\) so that \((l,l') \in C\) and
  \((l',l'') \in C'\). We conclude that \(l' \in NC(\sigma)\) and that
  \(l'' \in NC'(NC(\sigma))\).
\end{proof}

\begin{corollary}
  Let \(k\) be a field. The persistent homology \(H_*(N\Lambda)\) of
  \(N \Lambda\) with coefficients in \(k\) is functorial in the sense
  that it is a functor from the category of dissimilarities to
  the category of persistence modules over \(k\).
\end{corollary}

\subsection{Interleaving}
\label{sec:interleaving}

Here we present a notion of interleaving inspired by Bauer and
Lesnik \cite{MR3333456}.
\begin{definition}
  Let \(C\) and \(C'\) be filtered objects in a category \(\CC\) and
  let \(\alpha \colon \rnonneg \to \rnonneg\) be a translation
  function.
  \begin{enumerate}
  \item 
    A morphism \(G \colon C \to C'\) is an {\em
      \(\alpha\)-interleaving} if for every \(t \in \rstar\) there 
    exists a morphism \(F_t \colon C'(t) \to C({\alpha(t)})\) in
    \(\CC\) such
    that 
    \begin{displaymath}
      \alpha_{*C}(t) = F_t \circ G(t) \quad \text{and} \quad
      \alpha_{*C'}(t) = G(\alpha(t)) \circ F_t.
    \end{displaymath}
  \item 
    We say that \(C\) and \(C'\) are \(\alpha\)-interleaved if there
    exists an \(\alpha\)-interleaving \(G \colon C \to C'\).
  \end{enumerate}
\end{definition}

Suppose we are in the
situation that we have an inclusion \(K \subseteq L\) of filtered
simplicial complexes and that we are able to compute the filtration
value of simplices in \(L\), but we have no constructive way of
computing the filtration value of simplices in \(K\). In this
situation we can sometimes
construct at filtered simplicial sub-complex \(K'\) of \(L\) with
\(K'_{\infty} = 
K_\infty\). If the inclusion \(K \subseteq L\) is an
\(\alpha\)-interleaving, then the following lemma implies that also
the inclusion \(K' \subseteq L\) is an \(\alpha\)-interleaving. 
This happens for example when \(L\) is a \v{C}ech complex.
\begin{lemma}\label{sandwichinterleaving}
  Let \(C\), \(C'\) and \(C''\) be filtered objects in a category
  \(\CC\) and let \(\alpha \colon \rnonneg \to \rnonneg\) be a
  translation function. Let \(G \colon C \to C'\) and
  \(G' \colon C'\to C''\) be morphisms of filtered objects. Suppose
  that \(G'(t) \colon C'(t) \to C''(t)\) is a monomorphism for every
  \(t \in \rstar\) and that the composition \(G' G \colon C \to C''\)
  is an \(\alpha\)-interleaving. Then also \(G' \colon C' \to C''\) is
  an \(\alpha\)-interleaving.
\end{lemma}
\begin{proof}
  Let \(t \in \rstar\), and pick
  \(E_t \colon C''(t) \to C(\alpha(t))\) such that
  \begin{displaymath}
    \alpha_{*C}(t) = E_t \circ (G' G)(t) \quad \text{and} \quad
    \alpha_{*C''}(t) = (G' G)(\alpha(t)) \circ E_t. 
  \end{displaymath}
  Defining \(F_t = E_t G'(t) \colon C'(t) \to C(\alpha(t))\) the above
  relations imply that
  \begin{displaymath}
    \alpha_{*C}(t) = E_t \circ (G' G)(t) = (E_t G'(t)) \circ G(t) = F_t
    \circ G(t)
  \end{displaymath}
  and
  \begin{align*}
    G'(\alpha(t)) \circ \alpha_{*C'}(t) &= \alpha_{*C''}(t) \circ
    G'(t) \\
    &= (G' G)(\alpha(t)) \circ E_t \circ G'(t) \\
    &= G'(\alpha(t)) \circ G(\alpha(t)) \circ F_t.
  \end{align*}
  Then also \(\alpha_{*C'}(t) = G(\alpha(t)) \circ F_t\), since
  \(G'(\alpha(t))\) is a monomorphism.
\end{proof}

The following results are analogues of \cite[Proposition 2.2.11 and
Proposition 2.2.13]{MR3413628}.
\begin{lemma}[Functoriality]
  \label{inducedinterleaving}
  Let \(C\) and \(C'\) be filtered objects in a category \(\CC\), let
  \(\alpha \colon \rnonneg \to \rnonneg\) be a translation function
  and let \(H \colon \CC \to \DD\) be a functor.  If \(C\) and \(C'\)
  are \(\alpha\)-interleaved, then the filtered objects \(H C\) and
  \(H C'\) in \(\DD\) are \(\alpha\)-interleaved.
\end{lemma}
\begin{lemma}[Triangle inequality]
  Let \(G \colon C \to C'\) be an \(\alpha\)-interleaving and let
  \(G' \colon C' \to C''\) be an \(\alpha'\)-interleaving of filtered
  objects in a category \(\CC\). Then the composition
  \(G'' = G' G \colon C \to C''\) is an
  \(\alpha \alpha'\)-interleaving.
\end{lemma}
\begin{proof}
  Let \(\alpha'' = \alpha \alpha'\) and let \(t \in \rstar\).  By
  definition there exist morphisms
  \(F_{\alpha' (t)} \colon C'(\alpha' (t)) \to C(\alpha(\alpha
  '(t)))\) and \(F'_t \colon C''(t) \to C'(\alpha'(t))\) so that
  \begin{displaymath}
    \alpha_{*C}(\alpha' (t)) = F_{\alpha' (t)} \circ G(\alpha ' (t)) \quad \text{and} \quad
    \alpha_{*C'}(\alpha' (t)) = G(\alpha(\alpha ' (t))) \circ F_{\alpha' (t)}.
  \end{displaymath}
  and
  \begin{displaymath}
    \alpha'_{*C'}(t) = F'_t \circ G'(t) \quad \text{and} \quad
    \alpha'_{*C''}(t) = G'(\alpha'(t)) \circ F'_t.
  \end{displaymath}
  Let \(F''_t = F_{\alpha' (t)} F'_t \colon C''_t \to C_{\alpha \alpha'
    (t)}\).

  The above relations imply that the right hand triangles in the diagram
  \begin{displaymath}
  \begin{tikzcd}
    C(t) \arrow[rr, "G(t)"] \arrow[drr, "\alpha'_{*C}(t)" description]
    \arrow[ddrrrr, bend right, "\alpha''_{*C}(t)"']
    && C'(t) \arrow[rr, "G'(t)"]  \arrow[drr,
    "\alpha'_{*C'}(t)" description] 
    && C''(t) \arrow[d, "F'_t"] \\
    && C(\alpha'(t)) \arrow[rr, "G(\alpha'(t))"] 
    \arrow[drr, "\alpha_{*C}(\alpha'(t))" description] 
    &&
    C'(\alpha'(t)) \arrow[d, "F_{\alpha'(t)}"]\\
    && && C({\alpha \alpha'(t)})
  \end{tikzcd}    
  \end{displaymath}
  commute. The quadrangle in the above diagram commutes since \(G\) is
  a natural transformation and commutativity of the left hand triangle
  follows directly from the definition of the definition of the
  \(\alpha''\)-unit. We conclude that
  \begin{displaymath}
    \alpha''_{*C}(t) = F''_t \circ G''(t).
  \end{displaymath}

  The above relations also imply that the upper triangles in the diagram
  \begin{displaymath}
  \begin{tikzcd}
    C''(t) 
    \arrow[rr, "F'_t"]
    \arrow[drr, "\alpha'_{*C''}(t)" description]
    \arrow[ddrrrr, bend right, "\alpha''_{*C''}(t)"']
    &&
    C'(\alpha' (t)) 
    \arrow[rr, "F_{\alpha' (t)}"]
    \arrow[drr, "\alpha_{*C'}(\alpha' (t))" description]
    \arrow[d, "G'(\alpha' (t))"]
    &&
    C(\alpha \alpha' (t)) 
    \arrow[d, "G(\alpha \alpha' (t))"]
    \\
    &&
    C''(\alpha' (t)) 
    \arrow[drr, "\alpha'_{*C''}(\alpha' (t))" description]
    &&
    C'(\alpha \alpha' (t)) 
    \arrow[d, "G'(\alpha \alpha' (t))"]
    \\
    && 
    &&
    C''(\alpha \alpha' (t))
  \end{tikzcd}    
  \end{displaymath}
  commute. The quadrangle in the above diagram commutes since \(G'\)
  is a natural transformation and commutativity of the left hand
  triangle follows directly from the definition of the definition of
  the \(\alpha''\)-unit. We conclude that
  \begin{displaymath}
    \alpha''_{*C''}(t) = G''(\alpha''(t)) \circ F''_t.
  \end{displaymath}
\end{proof}

We find that the following lemma justifies our definition of
\(\alpha\)-interleaving.
\begin{lemma}
  Let \(\alpha \colon \rnonneg \to \rnonneg\) be a translation
  function and let \(V\) and \(V'\) be persistence modules.  A
  morphism \(G \colon V \to V'\) of persistence modules is an
  \(\alpha\)-interleaving if and only if both \(\ker G\) and
  \(\coker G\) are \(\alpha\)-trivial.
\end{lemma}
\begin{proof}
  Suppose first that \(G\) is an \(\alpha\)-interleaving.  Fix \(t\)
  and pick \(F_t \colon V'(t) \to V(\alpha(t))\) so that
  \(\alpha_{V*}(t) = F_t \circ G(t)\) and
  \(\alpha_{V'*}(t) = G(\alpha(t)) \circ F_t\).  Then
  \(v \in \ker G(t)\) implies
  \begin{displaymath}
    \alpha_{\ker G*} v = \alpha_{V*} v =F_t (G(t)(v)) = 0.
  \end{displaymath}
  Similarly, if
  \( v' + \im G(t) \in \coker G(t) \), then 
  \[\alpha_{\coker G*} (v'
  + \im G(t)) = \alpha_{V'*} (v') + \im G(\alpha(t)) = 0\]
  since
  \(\alpha_{V'*}(v') =  G(\alpha(t))(F_tv') \in \im G(\alpha(t))\).
  Thus \(\ker G\) and \(\coker G\) are
  \(\alpha\)-trivial.
  
  Conversely, suppose that \(\ker G\) and \(\coker G\) are
  \(\alpha\)-trivial and fix \(t\). Choose a basis \(e_1, \dots, e_a\)
  for \(\ker G(t)\) and choose \(f_1,\dots, f_b\) so that
  \begin{displaymath}
    e_1, \dots, e_a, f_1, \dots, f_b
  \end{displaymath}
  is a basis for \(C(t)\). Note that \(G(t)(f_1), \dots, G(t)(f_b)\)
  are linearly independent in \(C'(t)\) and choose \(g_1,\dots, g_c\)
  so that
  \begin{displaymath}
    G(t)(f_1), \dots, G(t)(f_b), g_1, \dots, g_c
  \end{displaymath}
  is a basis for \(C'(t)\). We use this basis to define \(F_t \colon
  C'(t) \to C(\alpha(t))\) as follows:
  On basis elements of the form \(G(t)(f_i)\) we define
  \begin{displaymath}
    F_t(G(t)(f_i)) = \alpha_{C*}(t)(f_i).
  \end{displaymath}
  Now consider basis elements of the form \(g_i\).
  Since \(\alpha_{\coker G*} = 0\) we know that \(\alpha_{C'*}(t)(g_i)
  \in \im G(\alpha(t))\). We choose \(c_i \in C(\alpha(t))\) so that
  \begin{displaymath}
    G(\alpha(t))(c_i) = \alpha_{C'*}(t)(g_i)
  \end{displaymath}
  and define
  \begin{displaymath}
    F_t(g_i) = c_i.
  \end{displaymath}
  Since \(\alpha_{\ker G} = 0\) we have
  \begin{displaymath}
    \alpha_{C*}(t)(e_i) = 0 = F_t(G(e_i)),
  \end{displaymath}
  so \(F_t G(t) = \alpha_{C*}(t)\). On the other hand, the equation
  \begin{displaymath}
    G(\alpha t) (F_t(G(t)(e_i))) = G(\alpha t) (\alpha_{C*}(e_i)) = \alpha_{C'*}(G(t)(e_i))
  \end{displaymath}
  shows that \(\alpha_{C'*}(t) = G(\alpha t) F_t\).
\end{proof}

\section{Matchings}
\label{sec:matching}

This presentation of matchings follows \cite{2018arXiv180106725H}.
\begin{definition}
  The set of {\em persistence intervals} is the set \(\upper\) of
  intervals in \(\rstar\).
\end{definition}
We write \(\overline{a}\) for the closure of an interval
\(a \in \upper\). Note that \(\overline a\) is determined by the end
points of the interval \(a\).
\begin{definition}
  A {\em persistence diagram} consists of a set \(X\) and a function
  \(p \colon X \to \upper\) from \(X\) to the set of persistence
  intervals. We refer to the elements of \(X\) as {\em persistence
    classes}. 
\end{definition}
\begin{definition}
  A {\em matching} \(R\) of two persistence diagrams
  \(p \colon X \to \upper\) and \(p' \colon X' \to \upper\) consists
  of a relation \(R \subseteq X \times X'\) with the property that the
  compositions
  \begin{displaymath}
    \pi_1 \colon R  \to X, \quad \pi_1(x,x') = x 
  \end{displaymath}
  and
  \begin{displaymath}
    \pi_2 \colon R  \to X', \quad \pi_2(x,x') = x' 
  \end{displaymath}
  with the inclusion of \(R\) in \(X \times X'\) and the projections
  to \(X\) and \(X'\) respectively are injective with \(p \circ \pi_1
  = p' \circ \pi_2\). We say that a
  persistence class \(x \in X\) is {matched by \(R\)} if there exists
  a persistence class \(x' \in X'\) so that \((x,x') \in
  R\). Similarly we say that a persistence class \(x' \in X'\) is
  {matched by \(R\)} if there exists a persistence class \(x \in X\)
  so that \((x,x') \in R\).
\end{definition}
\begin{definition}
  Let \(J\) be a subset of \(\rstar\). The
  persistence module \(k(J)\) has values
  \begin{displaymath}
    k(J)(t) =
    \begin{cases}
      k & \text{if \(t \in J\)} \\
      0 & \text{otherwise}
    \end{cases}
  \end{displaymath}
  and structure maps equal to identity maps whenever possible.
\end{definition}
\begin{definition}
  Let \(\alpha \colon \rnonneg \to \rnonneg\) be a translation
  function and let \(p \colon X \to \upper\) be a persistence diagram.
  We say that a persistence class \(x \in X\) is {\em
    \(\alpha\)-trivial} if the persistence module \(k(p(x))\) is
  \(\alpha\)-trivial. Otherwise we say that \(x\) is {\em
    \(\alpha\)-nontrivial}.
\end{definition}
Note that if \(p(x)\) has end points \(b < d\), then \(p(x)\) is
\(\alpha\)-trivial if and only if \(d \le \alpha(b)\).

\begin{definition}
  Let \(V\) be a persistence module over a field \(k\). We say that
  \(p \colon X \to \upper\) is {\em a  
  persistence 
  diagram of \(V\)} if there exists an isomorphism of the form
  \begin{displaymath}
    V \cong \bigoplus_{x \in X} k(p(x)).
  \end{displaymath}
\end{definition}
\begin{definition}
  The category of pointwise finite dimensional persistence modules
  over the field \(k\) is the full subcategory of the category of
  persistence modules \(V\) over \(k\) with \(V_t\) finite dimensional
  for every \(t \in \rstar\).
\end{definition}

We restate the decomposition theorem for pointwise finite-dimensional persistence modules \cite[Theorem 1.1]{MR3323327} in our notation. 
\begin{theorem}
  Let \(k\) be a field. Every pointwise finite dimensional persistence
  module over \(k\) has a persistence diagram.
\end{theorem}
We now state the generalized induced matching theorem \cite[Theorem
6.1]{MR3333456} and \cite[Theorem 3.2]{2018arXiv180106725H}. In order
to do this we use the generalized inverse function of a translation
function from Definition~\ref{generalizedinverse}.
\begin{theorem}
  \label{matchingthm}
  There exists a function \(\chi \colon \Pers \to \match\) from the
  set of morphisms of pointwise finite dimensional persistence modules
  over the field \(k\) to the set of matchings with the following
  properties: Let \(f \colon V \to V'\) be a morphism of pointwise
  finite persistence
  modules and let \(\chi(f)\) be of the form
  \begin{displaymath}
    \chi(f) \colon (X \xto p E) \to (X' \xto {p'} E),
  \end{displaymath}
  that is, 
  \begin{displaymath}
    \chi(f) \subseteq X \times X'. 
  \end{displaymath}
  Assume that \(f\) is an \(\alpha\)-interleaving and that
  \((x,x') \in \chi(f)\) with \(\overline{p(x)} = [b,d]\) and
  \(\overline{p'(x')} = [b',d']\). Then the following holds:
  \begin{enumerate}
  \item \(b' \le b < d' \le d\) and
  \item \(b \le \alpha(b')\) and
  \item \(\alpha^{\leftarrow}(d) \le d'\).
  \end{enumerate}
  Moreover all \(\alpha\)-nontrivial persistence classes of \(X\) and
  \(X'\) are matched by \(\chi(f)\).
\end{theorem}
In the above situation, if \(\alpha\) is bijective, then (3) is
equivalent to
\begin{displaymath}
  d \le \alpha(d').
\end{displaymath}
If we further assume that \(x'\) is \(\alpha\)-nontrivial, then
\(\alpha(b') < d'\) and the point \((b,d)\) lies in the box with
corners \((b',d')\) and \((\alpha(b'), \alpha(d'))\).  Conversely, if
\(\alpha\) is bijective and \(x\) is \(\alpha\)-nontrivial, then
\(\alpha(b) < d\) and the point \((b',d')\) lies in the box with
corners \((\alpha^{\leftarrow} b, \alpha^{\leftarrow} d)\) and
\((b,d)\).
\section{Truncated Nerves}
\label{sec:truncated}
\begin{definition}
  Let \(\Lambda \colon L \times W \to \rstar\) be a 
  dissimilarity and let \(\alpha \colon \rnonneg \to \rnonneg\) be a
  translation function. We say that a function
  \(T \colon L \to \rstar\) is an {\em \(\alpha\)-truncation function
    for \(\Lambda\)} if for all \(t \in \rstar\) and all \(l \in L\)
  there exists \(l' \in L\) so that for all \(w \in W\) with
  \(\Lambda(l,w) < t\) we have that \(\Lambda(l',w) < \alpha(t)\) and
  \(\Lambda(l',w) < T(l')\).
\end{definition}
\begin{definition}
  \label{truncatedDowkerDissimilarity}
  Let \(\Lambda \colon L \times W \to \rstar\) be a
  dissimilarity, let \(\alpha \colon \rnonneg \to \rnonneg\) be a
  translation function and let \(T \colon L \to \rstar\) be an
  \(\alpha\)-truncation function for \(\Lambda\).  The {\em
    \(T\)-truncation of \(\Lambda\)} is the dissimilarity
  \(\Gamma \colon L \times W \to \rstar\) defined by
  \begin{displaymath}
    \Gamma(l,w) =
    \begin{cases}
      \Lambda(l,w) & \text{if \(\Lambda(l,w) < T(l)\)} \\
      \infty & \text{otherwise}.
    \end{cases}
  \end{displaymath}
\end{definition}
\begin{proposition}
  \label{interleavingoftruncation}
  Let \(\Lambda \colon L \times W \to \rstar\) be a
  dissimilarity, let \(\alpha \colon \rnonneg \to \rnonneg\) be a
  translation function and let \(T \colon L \to \rstar\) be an
  \(\alpha\)-truncation function for \(\Lambda\).  Let \(\Gamma\) be
  the \(T\)-truncation of \(\Lambda\). Then, in the homotopy category
  of topological spaces, the inclusion of the nerve
  \(N\Gamma\) of \(\Gamma\) in the nerve \(N\Lambda\) of \(\Lambda\)
  is an \(\alpha\)-interleaving.
\end{proposition}
\begin{proof}
  It suffices, for every \(t \in \rstar\), to find a map \(f_t \colon
  N\Lambda_t \to N\Gamma_{\alpha(t)}\) so that the following diagrams
  commute up to homotopy:
  \begin{displaymath}
  \begin{tikzcd}
    N\Gamma_t \arrow[r] \arrow[dr, swap, "N\Gamma_{t \le \alpha(t)}"] & 
    N\Lambda_t \arrow[d, "f_t"] \\
    &
    N\Gamma_{\alpha(t)}
  \end{tikzcd}    
  \end{displaymath}
  and
  \begin{displaymath}
  \begin{tikzcd}
    N\Lambda_t \arrow[r, "f_t"] \arrow[dr, swap, "N\Lambda_{t \le \alpha(t)}"] & 
    N\Gamma_{\alpha(t)} \arrow[d] \\
    &
    N\Lambda_{\alpha(t)}.
  \end{tikzcd}    
  \end{displaymath}
  Fix \(t\) and choose a function 
  \(f_t \colon L \to L\) so that for every \(l \in L\) with
  \(\Lambda(l,w) < t\) the inequalities \(\Lambda(f_t(l),w) < \alpha(t)\) and
  \(\Lambda(f_t(l),w) < T(f_t(l))\) hold. 

  Below we first show that \(f_t\) induces a simplicial map
  \begin{displaymath}
    f_t \colon N\Lambda_t \to N\Gamma_{\alpha(t)}.
  \end{displaymath}
  That is, we show that if \(\sigma \in N\Lambda_t\), then
  \(f_t(\sigma) \in N\Gamma_{\alpha(t)}\). Next we show that
  \(f_t(\sigma) \cup \sigma \in N\Lambda_{\alpha(t)}\) so that the
  lower of the above displayed diagrams commutes up to homotopy. We
  will finish by showing that if \(\sigma \in N\Gamma_t\), then
  \(f_t(\sigma) \cup \sigma \in N\Gamma_{\alpha(t)}\) so that also the
  upper of the above displayed diagrams commutes up to homotopy.

  Let \(\sigma \in N\Lambda_t\) and pick \(w \in W\) so that
  \(\Lambda(l,w) < t\) for every \(l \in \sigma\). Then, for every \(l
  \in \sigma\) we have \(
  \Lambda(f_t(l),w) < \alpha(t)\) and \(\Lambda(f_t(l),w) <
  T(f_t(l)))\) so in particular
  \(\Gamma(f_t(l),w) = \Lambda(f_t(l),w) < \alpha(t)\). This implies
  both that \(f_t(\sigma) \in N\Gamma_{\alpha(t)}\) and 
  that \(f_t(\sigma) \cup \sigma \in N\Lambda_{\alpha(t)}\). Finally,
  if \(\sigma \in N\Gamma_t\) and we pick 
  \(w \in W\) so that
  \(\Gamma(l,w) = \Lambda(l,w) < t\) for every \(l \in \sigma\), then
  the above argument also implies that \(f_t(\sigma) \cup \sigma \in N\Gamma_{\alpha(t)}\).
\end{proof}
\begin{example}
  Let \(\Lambda \colon L \times W \to \rstar\) be a
  dissimilarity and let \(\alpha \colon \rnonneg \to \rnonneg\) be a
  translation function. 
  Given \(l,l' \in L\), let
  \begin{displaymath}
    P(l',l) = \{ \Lambda(l',w) \, \mid \,
    w \in W \text{ with }
    \alpha(\Lambda(l,w)) \le \Lambda(l',w) \}.
  \end{displaymath}
  Given a base point \(l_0 \in L\) the
  {\em cover dissimilarity
    \(\Lambda^{\alpha} \colon L \times L \to \rstar\) for \(\Lambda\)
    with respect to \(\alpha\)} is given by
  \begin{displaymath}
    \Lambda^{\alpha}(l',l) =
    \begin{cases}
      0 & \text{if \(l = l'\)}\\
      \infty & \text{if \(l \ne l'\) and \(l' = l_0\)}\\
      \sup(P(l',l) \cup \{0\}) & \text{if \(l \ne l'\) and \(l' \ne l_0\)}.
    \end{cases}
  \end{displaymath}
  If \(L\) is finite and \(<\) is a total order on \(L\) with \(l_0\)
  as minimal 
  element, we define the
  {\em \(\alpha\)-insertion radius \(\lambda^\alpha(l)\) of \(l\in L\)} as
  \begin{displaymath}
    \lambda^\alpha(l) =
    \begin{cases}
      \infty & \text{if \(l = l_0\)}\\
      \sup_{k \ge l} \inf_{l' < l} \Lambda^{\alpha}(l', k) 
      & \text{if \(l \ne l_0\)}
    \end{cases}
  \end{displaymath}

  Given \(l \in L\) with \(l \ne l_0\) and \(t \in \rstar\) with
  \(\alpha(t) > 0\), pick 
  \(l' \in L\) minimal with \(\Lambda^{\alpha}(l',l) <
  \alpha(t)\). (Such an \(l'\) exists since
  \(\Lambda^{\alpha}(l,l) = 0\) and \(\Lambda^{\alpha}(l_0,l) =
  \infty\)). 
  Let \(w \in W\) with \(\Lambda(l,w) < t\). Then either
  \(\Lambda(l',w) > \Lambda^{\alpha}(l',l)\) or
  \(\Lambda(l',w) \le \Lambda^{\alpha}(l',l) < \alpha(t)\). If
  \(\Lambda(l',w) > \Lambda^{\alpha}(l',l)\), then
  \begin{displaymath}
    \Lambda(l', w) < \alpha(\Lambda(l,w)) \le \alpha(t). 
  \end{displaymath}
  Thus, \(\Lambda(l,w) < t\) implies \(\alpha(t) > \Lambda(l',
  w)\). Since \(\infty  = \lambda^\alpha(l_0) \ge \alpha(t)\) and 
  \begin{displaymath}
    \lambda^\alpha(l') = \sup_{k \ge l'} \inf_{l'' < l'}
    \Lambda^{\alpha}(l'', k) \ge 
    \inf_{l'' < l'} \Lambda^{\alpha}(l'', l) \ge \alpha(t) ,
  \end{displaymath}
  the function \(\lambda^\alpha \colon L \to \rstar\) is an
  \(\alpha\)-truncation function for \(\Lambda\).
\end{example}

\begin{definition}\label{fps}
  Given a dissimilarity of the form
  \(\Lambda \colon L \times L \to \rstar\) with \(L\) finite a {\em
    farthest point sample} for \(\Lambda\) is a total order \(<\) on
  \(L\) with minimal element \(l_0\) so that for \(l \ne l_0\) we have
  \begin{displaymath}
    \inf_{l' < l} \Lambda(l',l) = \sup_{l'' \ge l} \inf_{l' < l} \Lambda(l', l'').
  \end{displaymath}
  The {\em insertion radius} of \(l \in L\) with respect to the total
  order \(<\) is
  \begin{displaymath}
    \lambda(l) =
    \begin{cases}
      \infty & \text{if \(l=l_0\)} \\
      \inf_{l' < l} \Lambda(l',l) & \text{otherwise}. 
    \end{cases}
  \end{displaymath}
\end{definition}
For \(\Lambda\) as in Definition~\ref{fps} a farthest point sample
\(L = \{l_0 < \dots < l_n\}\) can be produced recursively starting
from an initial point \(l_0\). When \(l_0,\dots,l_k\) have been
produced, we choose \(l_{k+1}\) so that
\begin{displaymath}
  \inf_{l' \in \{l_0,\dots,
    l_k\}} \Lambda(l',l_{k+1}) = \sup_{l'' \notin \{l_0,\dots,
    l_k\}} \inf_{l' \in \{l_0,\dots,
    l_k\}} \Lambda(l', l''). 
\end{displaymath}
Note that 
\begin{displaymath}
  \lambda(l) =
  \begin{cases}
    \infty & \text{if \(l = l_0\)} \\
    \sup_{k \ge l} \inf_{l' < l} \Lambda(l', k) & \text{otherwise}.
  \end{cases}
\end{displaymath}

\begin{example}
  Let \(d \colon W \times W \to \rstar\) be a metric, let \(L\) be a
  finite subset of \(L\) and let
  \(\Lambda \colon L \times W \to \rstar\) be the restriction of \(d\)
  to the subset \(L \times W\) of \(W \times W\). Let
  \(\Lambda^L \colon L \times L \to \rstar\) be the restriction of
  \(\Lambda\) to the subset \(L \times L\) of \(L \times W\) and let
  \(L = \{l_0 < \dots < l_n\}\) be a farthest point sampling for
  \(\Lambda^L\). We write \(\lambda^L(l) = \lambda(l)\) for the
  corresponding insertion radius. Let \(c > 1\) and let
  \(\alpha \colon \rnonneg \to \rnonneg\) be the translation function
  \(\alpha(t) = ct\).

  If \(\alpha\Lambda(l,w) \le \Lambda(l',w)\), then the triangle
  inequality for \(d\) implies that
  \(\Lambda(l',w) \le \Lambda^L(l',l) + \Lambda(l,w)\) and therefore
  \(\Lambda(l,w) \le \Lambda^L(l',l)/(c-1)\). This, together with the
  triangle inequality for \(d\) implies that
  \begin{displaymath}
    \Lambda(l',w) \le \Lambda^L(l',l) + \Lambda(l,w) \le
    \Lambda^L(l',l) + \frac{\Lambda^L(l',l)}{(c-1)} =
    \frac{c\Lambda^L(l', 
    l)}{(c-1)}. 
  \end{displaymath}
  From this consideration we can conclude that
  \begin{displaymath}
    \Lambda^{\alpha}(l',l) \le \frac{c\Lambda^L(l',
    l)}{(c-1)} 
  \end{displaymath}
  and that
  \begin{displaymath}
    \lambda^{\alpha}(l) \le \frac{c\lambda^L(l)}{(c-1)}.
  \end{displaymath}
  Since \(\lambda^{\alpha}\) is an \(\alpha\)-truncation function of
  \(\Lambda\), so is the function \(T(l) = {c\lambda^L(l)}/{(c-1)}\).
\end{example}

There exist many truncation functions for a given translation function
\(\alpha\). We have not succeeded in finding a class of truncation
functions for \(\alpha\) that are practical to implement and produces a
smallest possible simplicial 
complex under this constraint. We leave this as a problem for further
investigation. If the goal is merely to construct a dissimilarity
whose Dowker nerve is small the amount of possibilities is even
bigger. 

\section{Sparse Filtered Nerves}
\label{sec:dnerves}
\begin{definition}
  Let \(\Lambda \colon L \times W \to \rstar\) be a 
  dissimilarity and let \(R \colon L \to \rstar\) and \(\varphi \colon
  L \to L\) be functions. We say that \(l \in L\) is a {\em slope point}
  if \(R(l) < R(l')\) for every \(l' \in \varphi^{-1}(l)\). The
  {\em \((R, \varphi)\)-nerve of \(\Lambda\)} is the filtered simplicial complex 
  \(N(\Lambda, R, \varphi)\) with  
  \(N(\Lambda, R, \varphi)(t)\) consisting of all \(\sigma \in N\Lambda_t\)
  such that there exists \(w \in W\) satisfying:
  \begin{enumerate}
  \item \(\Lambda(l,w) < t\) for all \(l \in \sigma\).
  \item \(\Lambda(l,w) \le R(l')\) for all \(l, l' \in \sigma\) and
  \item \(\Lambda(l,w) < R(l)\) for all slope points \(l\) in \(\sigma\). 
  \end{enumerate}
\end{definition}
\begin{definition}
  A function \(\varphi \colon L \to L\) is a {\em parent function} if
  \(\varphi^n(l) = l\) for \(n > 0\) implies \(\varphi(l) = l\).  
\end{definition}
Note that \(\varphi \colon L \to L\) is a parent function if and only
if the directed graph with \(L\) as set of
nodes and \(E(\varphi) = \{(\varphi(l), l) \, \mid \, l \in L, \varphi(l) \neq l\}\) as set of edges is acyclic.
\begin{definition}
\label{sparsificationstructure}
  Let \(\Lambda \colon L \times W \to
  \rstar\) be a
  dissimilarity and let \(\varphi \colon L \to L\) be a parent function.
  We say that a function \(R \colon 
  L \to \rstar\) is a {\em restriction function for \(\Lambda\)
    relative to \(\varphi\)} if the following holds: 
  \begin{enumerate}
  \item For all \((l,w)\in L \times W\) with \(\Lambda(l,w) <
    \Lambda(\varphi(l),w)\)
    we have
    \(\Lambda(\varphi(l),w) \le R(l)\).
  \item For every \(l\in L\) we have
    \(R(\varphi(l)) \ge R(l)\).
  \item If \(\varphi(l) = l\), then \(R(l) = \infty\).
  \end{enumerate}
\end{definition}
\begin{definition}
  \label{canonicalrestriction}
  Let \(\Lambda \colon L \times W \to \rstar\) be a
  dissimilarity and let \(\varphi \colon L \to L\) be a parent function.
  Given \(l,l'\in L\) let
  \begin{displaymath}
    P(l,l') = \{ \Lambda(l',w) \, \mid \,
      w \in W \text{ with } \Lambda(l,w) < \Lambda(l',w) \}
  \end{displaymath}
  and define
  \begin{displaymath}
    \rho(l,l') =
    \begin{cases}
      \sup P(l,l') & \text{if \(P(l,l')\) is non-empty} \\
      0 & \text{if \(P(l,l') = \emptyset\).} \\
    \end{cases}
  \end{displaymath}
  The {\em \((\Lambda, \varphi)\)-restriction} \(R({\Lambda}, \varphi)
  \colon L \to \rstar\) is defined in several steps. First define \(R'
  \colon L \to \rstar\) by
  \begin{displaymath}
    R'(l) =
    \begin{cases}
      \rho(l, \varphi(l)) & \text{ if \(\varphi(l) \ne l\)} \\
      \infty & \text{ if \(\varphi(l) = l\).}
    \end{cases}
  \end{displaymath}
  Given \(l \in L\), let \(D(l)\) be the set of descendants of \(l\),
  that is, \(l' \in D(l)\) if and only if there exists \(m \ge 0\) so
  that \(l = \varphi^m(l')\). Next, we define \(R(\Lambda, \varphi)
  \colon L \to L\) by 
  \begin{displaymath}
    R(\Lambda, \varphi)(l) = \max_{l' \in D(l)} R'(l').
  \end{displaymath}
  Then, for every \(l\in L\) we have
  \(R(\Lambda, \varphi)(\varphi(l)) \ge R(\Lambda, \varphi)(l)\), and \(\varphi(l) = l\) implies
  \(R(\Lambda, \varphi)(l) = \infty\). Also, \(\Lambda(l,w) < \Lambda(\varphi(l), w)\)
  implies
  \begin{displaymath}
    \Lambda(\varphi(l),w) \le R'(l) \le R(\Lambda, \varphi)(l).
  \end{displaymath}
\end{definition}
\begin{proposition}
  \label{minimalityofcanonical}
  Let \(\Lambda \colon L \times W \to \rstar\) be a
  dissimilarity with \(L\) finite and let \(\varphi \colon L \to L\)
  be a parent function.
  Then the \((\Lambda,\varphi)\)-restriction \(R(\Lambda, \varphi)\) 
  is the minimal
  restriction function for \(\Lambda\)
  relative to \(\varphi\):
  If \(R\) is another restriction function for \(\Lambda\) relative to
  \(\varphi\),
  then \(R(\Lambda, \varphi)(l) \le R(l)\) for every \(l \in L\).
\end{proposition}
\begin{proof}
  In the notation of Definition \ref{canonicalrestriction} it suffices
  to show that \(\rho(l, \varphi(l)) \le R(l)\) for all \(l \in L\).
  We can assume that \(\varphi(l) \neq l\) because 
  otherwise \(\rho(l, \varphi(l)) = \infty = R(l)\). 
  Given \(l \in L\), if there exists a \(w \in W\) with 
  \begin{displaymath}
    \Lambda(l,w) < \Lambda(\varphi(l),w)
  \end{displaymath}
  we have \(\Lambda(\varphi(l),w) \le R(l)\). By construction of
  \(\rho\), this implies that
  \begin{displaymath}
    \rho(l, \varphi(l)) \le R(l).
  \end{displaymath}
  If no such \(w \in W\) with 
  \begin{displaymath}
    \Lambda(l,w) < \Lambda(\varphi(l),w)
  \end{displaymath}
  exists, then \(\rho(l, \varphi(l)) = 0 \le R(l)\).
\end{proof}

Proposition~\ref{minimalityofcanonical} shows that the
\((\Lambda, \varphi)\)-restriction function is the minimal restriction function
for \(\Lambda\) relative to \(\varphi\). In the following example we
propose a strategy for construction of a parent function.
\begin{example}
  Let \(\Lambda \colon L \times W \to \rstar\) be a
  dissimilarity with \(L\) finite.
  As in Definition \ref{canonicalrestriction}, given \(l,l'\in L\) let
  \begin{displaymath}
    P(l,l') = \{ \Lambda(l',w) \, \mid \,
      w \in W \text{ with } \Lambda(l,w) < \Lambda(l',w) \}
  \end{displaymath}
  and define
    \begin{displaymath}
    \rho(l,l') =
    \begin{cases}
      \sup P(l,l') & \text{if \(P(l,l')\) is non-empty} \\
      0 & \text{if \(P(l,l') = \emptyset\).} 
    \end{cases}
  \end{displaymath}
  Define \(R_0 \colon L \to \rstar\) by
  \begin{displaymath}
    R_0(l) = \inf \{ \rho(l,l') \, \mid \, l' \in L
    \text{ and } l' \ne l\}
  \end{displaymath}
  and let \(<\) be a total order on \(L\) with minimal element \(l_0\)
  so that \(R_0(l') > R_0(l)\)
  implies \(l' < l\). Given \(l\in L\) let
  \begin{displaymath}
    Q(l) = \{l' \in L \, \mid \, l' < l \text{ and } \rho(l,l') = R_0(l)\}.
  \end{displaymath}
  If \(Q(l)\) is non-empty we define
  \begin{displaymath}
    \varphi(l) = \min Q(l).
  \end{displaymath}
  Otherwise, that is, if \(Q(l)\) is empty, we let
  \begin{displaymath}
    R_1(l) = \inf \{ \rho(l,l') \, \mid \, l' < l\}.
  \end{displaymath} and define
  \begin{displaymath}
    \varphi(l) = \min (\{l' \in L \, \mid \, l'< l \text{ and }
    \rho(l,l') = R_1(l) \} \cup \{l_0\}).
  \end{displaymath}
  Since
  \(\varphi(l) \le l\) for every \(l \in L\) and \(<\) is 
  a total order on \(L\), the function \(\varphi \colon L \to L\) 
  is a parent function.  We define \(R \colon L  \to \rstar\) 
  to be the restriction function for \(\Lambda\) relative
  to \(\varphi\) constructed in Definition \ref{canonicalrestriction}.

\end{example}
In the following two examples
we show that the 
sparsifications from \cite{SparseDowker, Sheehy2013} also are
\((R, \varphi)\)-nerves and that therefore the
\((\Lambda, \varphi)\)-restriction results in smaller nerves.
\begin{example}[Parent restriction]
  In \cite{SparseDowker} we constructed the {\em sparse filtered nerve}
  \(N\Lambda\) of a dissimilarity \(\Lambda \colon L \times W
  \to \rstar\) with \(L\) finite. 
  In this example we describe a function \(\varphi \colon L \to L\)
  and a restriction \(R\) for \(\Lambda\) relative to \(\varphi\) so that
  the \((R, \varphi)\)-nerve of \(\Lambda\) is equal to the sparse Dowker
  nerve in \cite[Definition 38]{SparseDowker}.
  Given \(l \in L\) we let
  \begin{displaymath}
    W_l = \{ w \in W \, \mid \, \Lambda(l,w) < \infty\}
  \end{displaymath}
  and
  \begin{displaymath}
    \widetilde \tau(l) = \sup\{ \Lambda(l,w) \, \mid \, w \in W_l\}.
  \end{displaymath}
  Let \(l_0 \in L\) and define \(\tau
  \colon L \to \rstar\) as the function 
  \begin{displaymath}
    \tau(l) =
    \begin{cases}
      \infty & \text{if \(\{l = l_0\}\)} \\
      \widetilde \tau(l) & \text{otherwise}.
    \end{cases}
  \end{displaymath}
  Given \(l \in L\) we let
  \begin{displaymath}
    Q(l) = \{l' \, \mid \, \tau(l') > \tau(l)\}.
  \end{displaymath}
  If \(Q(l)\) is non-empty we pick \(l' \in Q(l)\) and define
  \(\varphi(l) = l'\). Otherwise we define \(\varphi(l) =
  l_0\). 
  The {\em parent restriction} \(R \colon L \to \rstar\) is defined by
  \begin{displaymath}
    R(l) = \tau(\varphi(l)).
  \end{displaymath}

  It is readily verified that the above structure satisfies is a
  sparsification function for \(\Lambda\) with respect to \(\varphi\).
  The \(R\)-nerve \(N(\Lambda, R, \varphi)\) is the sparse Dowker nerve
  introduced in \cite[Definition 38]{SparseDowker}
\end{example}

\begin{example}[Sheehy restriction]
  \label{Sheehyex}
  Let \(\Lambda \colon L \times W \to \rstar\) be a dissimilarity
  with \(\Lambda(l,w) < \infty\) for all \((l,w) \in L \times W\) and
  satisfying the triangle inequality. The specific case we have in mind
  is \(L = W\) and \(d \colon L \times L \to (0,\infty)\) a metric. Let
  \(\alpha \colon \rnonneg \to \rnonneg\) be a translation function of
  the form \(\alpha = \id + \beta\) for an order preserving function
  \(\beta \colon \rnonneg \to \rnonneg\) so that the function
  \(\rstar \to [\beta(0),\infty)\) taking \(t\) to \(\beta(t)\) is
  bijective. Note that this implies that
  \(\beta^{\leftarrow} \beta = \id\).

  Let \(\lambda \colon L \to \rstar\) be the canonical insertion
  function for \(\Lambda\) as defined in \cite{SparseDowker} and let
  \(\varphi \colon L \to L\) be the associated parent function. That is, for
  \(L = \{l_0, \dots, l_n\}\), the function \(\lambda \colon L \to
  \rstar\) is defined by
    \begin{displaymath}
      \lambda(l) =
      \begin{cases}
        \infty & \text{ if \(l = l_0\)} \\
        \sup_{w \in W} \inf_{l' \in \{l_0, \dots, l_{k-1}\} }  \Lambda(l', w) & \text{ if \(l = l_k\)}
      \end{cases}
    \end{displaymath}
    and \(\varphi(l)\) is the smallest element in \(L\) such that
    there exists \(w \in W\) with
    \(\Lambda(\varphi(l), w) = \lambda(l)\). 
    Recall that the
    \((\lambda, \beta)\)-truncation of \(\Lambda\) is the
    dissimilarity \(\Gamma \colon L \times W \to \rstar\) defined in \cite{SparseDowker}
    by
 \begin{displaymath}
    \Gamma(l,w) =
    \begin{cases}
      \Lambda(l,w) & \text{if \(\Lambda(l,w)
        \le \alpha \beta^{\leftarrow}\lambda(l)\) and \(\beta(0)
        \le \lambda(l)\)} \\ 
      \infty & \text{otherwise}.
    \end{cases}
  \end{displaymath} 
  
  The \((\lambda, \beta)\)-truncation of \(\Lambda\) is not a
  truncation of \(\Lambda\) as 
  defined in Definition~\ref{truncatedDowkerDissimilarity}. However,
  the dissimilarity \(\Gamma' \colon L \times W \to \rstar\)
  described by the formula
  \begin{displaymath}
    \Gamma'(l,w) =
    \begin{cases}
      \Lambda(l,w) & \text{if \(\Lambda(l,w)
        < \alpha \beta^{\leftarrow} \lambda(l)\) } \\
      \infty & \text{otherwise}
    \end{cases}
  \end{displaymath}
  is a truncated dissimilarity and is smaller than \(\Gamma\). 
  Since elements \(l \in L\) with \(\lambda(l) < \beta(0)\) do not
  contribute to the Dowker nerve of \(\Gamma\) we assume without loss
  of generality that \(\beta(0) \le \lambda(l)\) for every
  \(l \in L\).
  
  The {\em Sheehy restriction function} is
  \begin{displaymath}
    S \colon L \to \rstar, \quad S(l) = \alpha^2 \beta^{\leftarrow} \lambda(l)
  \end{displaymath}
  Our assumption on \(l_0\) implies that \(S(l_0) = \infty\).
  The {\em Sheehy parent function} \(\varphi \colon L \to L\) is
  defined as follows: first we define \(\varphi(l_0) = l_0\) for the
  minimal element \(l_0\) in \(L\). Next,
  given \(l \ne l_0\), choose \(w' \in W\) so that 
  \((l,w') \in T\) and use that \(\lambda\) is a \(\beta\)-insertion
  function to choose \(l' \in 
  L\) so that 
  \begin{displaymath}
    \Lambda(l',w') \le \beta \alpha \beta^{\leftarrow} \lambda(l) < \lambda(l'),
  \end{displaymath}
  and define \(\varphi(l) = l'\). Since
  \begin{displaymath}
    \lambda(l) \le \beta \beta^{\leftarrow} \lambda(l) \le 
    \beta \alpha \beta^{\leftarrow} \lambda(l) < \lambda(l')
  \end{displaymath}
  we have \(S(\varphi(l)) > S(l)\) for every \(l\in L\) with \(S(l) < \infty\).

  Given \(w \in W\) with \(\Gamma'(l,w) < \infty\) we have 
  \(\Lambda(l,w) < \alpha
    \beta^{\leftarrow} \lambda(l),\)
  so for \(l'\) and \(w'\) as above the triangle inequality gives
  \begin{displaymath}
    \Lambda(l', w) \le \Lambda(l', w') + \Lambda(l, w) <
    \beta \alpha \beta^{\leftarrow} \lambda(l) + \alpha
    \beta^{\leftarrow} \lambda(l) = \alpha^{2} \beta^{\leftarrow} \lambda(l).
  \end{displaymath}
  The inequality \(\Lambda(l', w) \le \alpha^{2}
  \beta^{\leftarrow} 
  \lambda(l)\) implies 
  \begin{displaymath}
    \beta \alpha^{\leftarrow} \Lambda(l',w) \le 
    \beta \alpha^{\leftarrow} \alpha^2
    \beta^{\leftarrow} \lambda(l) \le
    \beta \alpha
    \beta^{\leftarrow} \lambda(l) < \lambda(l').
  \end{displaymath}
  Since \(\varphi(l) = l'\) we can conclude that
  \begin{displaymath}
    \Gamma'(\varphi(l),w) = \Lambda(\varphi(l), w) \le \alpha^{2} \beta^{\leftarrow}
    \lambda(l) = S(l)
  \end{displaymath}
  for every \(l \in L\). We conclude that \(\Gamma'(l,w) < \infty\)
  implies \(\Gamma'(\varphi(l),w) \le S(l)\). 
\end{example}

\begin{theorem}\label{sparsificationtheorem}
  Let \(\Lambda
  \colon L \times W \to \rstar\) be a
  dissimilarity and 
  let \(R\) be a restriction function for \(\Lambda\) relative to
  \(\varphi \colon L \to L\). If \(L\) is
  finite, then for every \(t
  \in \rstar\) the 
  geometric realization of the inclusion 
  \begin{displaymath}
    \iota \colon N(\Lambda, R, \varphi)(t) \to  N\Lambda_t
  \end{displaymath}
  is a homotopy equivalence.
\end{theorem}
\begin{proof}
  Since \(R(l) = \infty\) whenever \(\varphi(l) = l\), the two
  complexes agree when when \(L\) is 
  of cardinality \(1\), 
  and thus the result holds in this case. 
  Let \(t \in \rstar\) and let \(n > 1\).
  Below we will show that if \(\Lambda \colon L \times W \to \rstar\) is a
  dissimilarity with \(L\) a set of cardinality \(n\)
  and \(R\) is a restriction function for \(\Lambda\) relative to
  \(\varphi\) so that the inclusion  
  \begin{displaymath}
    \iota \colon N(\Lambda, R)(t) \to  N\Lambda_t
  \end{displaymath}
  is not a homotopy equivalence, then there exists a
  dissimilarity
  \(\Lambda' \colon L' \times W \to \rstar\) with \(L'\) a set of
  cardinality \(n -1\)
  and \(R'\) a restriction function for \(\Lambda'\) relative to a
  function \(\varphi' \colon L' \to L'\) so that the inclusion 
  \begin{displaymath}
    \iota \colon N(\Lambda', R')(t) \to  N\Lambda'_t
  \end{displaymath}
  is not a homotopy equivalence. Negating this we obtain the inductive
  step implying that the result holds for all finite sets \(L\).

  As above, 
  let \(\Lambda \colon L \times W \to \rstar\) be a
  dissimilarity with \(L\) a set of cardinality \(n > 1\)
  and let \(R\) be a restriction function for \(\Lambda\) relative to
  \(\varphi\). Fix \(t \in 
  \rstar\) and pick 
  \(l_n \in L\) 
  so that firstly \(R(l_n) \le R(l)\) for every \(l \in L\) and
  secondly \(l_n\) is not in the image of \(\varphi \colon L \to
  L\). This is possible 
  since \(L\) is finite and \(R(\varphi(l)) \le R(l)\) for every \(l
  \in L\).
  Suppose that 
  \begin{displaymath}
    \iota \colon N(\Lambda, R)(t) \to  N\Lambda_t
  \end{displaymath}
  is not a homotopy equivalence. Then there exists \(l \in L\) with
  \(R(l) < t\) since otherwise the two complexes are obviously
  equal. In particular \(R(l_n) < t\).  

  Let \(L' = L \setminus \{l_n\}\) and let \(\Lambda' \colon L' \times W
  \to \rstar\) be the restriction of \(\Lambda\) to \(L' \times W
  \subseteq L \times W\). Further, let \(R' \colon L' \to
  \rstar\) be the restriction of \(R\) to \(L'\) 
  and let \(\varphi' \colon L' \to L\) be the
  restriction of \(\varphi\) to \(L'\). Since \(l_n\) is not in the
  image of \(\varphi\) we can consider
  \(\varphi'\) as a function \(\varphi' \colon L' \to
  L'\). Clearly \(R'\) is a
  restriction function for \(\Lambda'\) relative to \(\varphi'\).
  
  We define \(f_t
  \colon L \to L'\) by 
  \begin{displaymath}
    f_t(l) =
    \begin{cases}
      \varphi(l_n) & \text{if \(l = l_n\)}\\
      l & \text{otherwise.}
    \end{cases}
  \end{displaymath}

  Given \(\sigma \in N\Lambda_t\) we claim that \(\sigma \cup
  f_t(\sigma) \in N\Lambda_t\). If \(l_n \notin \sigma\), then this
  claim is trivially satisfied. In order to justify the claim when
  \(l_n \in \sigma\) we pick  
  \(w \in W\) with \(\Lambda(l,w) < t\) for every \(l \in \sigma\). If
  \(\Lambda(l_n,w) \ge \Lambda(\varphi(l_n),w)\) then
  \(\Lambda(\varphi(l_n),w) < t\) and \(\sigma \cup f_t(\sigma) \in
  N\Lambda_t\). Otherwise by part (1) of
  Definition~\ref{sparsificationstructure} the inequalities
  \(\Lambda(l_n,w) < t\) and \(\Lambda(l_n,w) < 
  \Lambda(\varphi(l_n),w)\) 
  imply that \(\Lambda(\varphi(l_n),w) \le R(l_n) < t\).
  We conclude that \(\sigma \cup
  f_t(\sigma) \in N\Lambda_t\) also in this situation.

  Next we claim that
  \begin{displaymath}
    \sigma \in N(\Lambda, R, \varphi)(t) 
    \quad \text{ implies } \quad
    \sigma \cup f_t(\sigma) \in N(\Lambda, R, \varphi)(t).
  \end{displaymath}
  Again we only 
  need to consider the case \(l_n \in \sigma\). We have already shown
  that \(\sigma \cup f_t(\sigma) \in N \Lambda_t\). 
  Pick
  \(w \in W\) with \(\Lambda(l,w) < t\), \(\Lambda(l,w) < R(l)\) and
  \(\Lambda(l,w) \le R(l')\) 
  for every \(l, l' \in \sigma\). Note 
  in
  particular that 
  \(\Lambda(l_n,w) < t\). 
  If
  \(\Lambda(l_n,w) \ge \Lambda(\varphi(l_n),w)\) then
  \begin{displaymath}
    R(l) \ge R(l_n) > \Lambda(l_n, w) \ge \Lambda(\varphi(l_n),w)
  \end{displaymath}
  for all \(l \in L\),
  so \(\sigma \cup f_t(\sigma) \in N(\Lambda, R)(t)\). 
  On the other hand, if \(\Lambda(l_n,w) < \Lambda(\varphi(l_n),w)\)
  then by
  (1) of Definition~\ref{sparsificationstructure}
  the inequality \(\Lambda(l_n,w) < t\) 
  implies
  \(\Lambda(\varphi(l_n),w) \le R(l_n)
  \le R(l)\) for all \(l \in L\).
  If \(\varphi(l_n)\) is a slope point, then 
  \(\Lambda(\varphi(l_n),w) \le R(l_n) <
  R(\varphi(l_n))\).
  We conclude that \(\sigma \cup f_t(\sigma) \in N(\Lambda,
  R, \varphi)(t)\).

  We can now conclude that the function \(f_t \colon L
  \to L'\) 
  defines simplicial
  maps
  \begin{displaymath}
    f_t \colon N\Lambda_t \to N\Lambda'_t
  \end{displaymath}
  and
  \begin{displaymath}
    f_t \colon N(\Lambda, R, \varphi)(t) \to N(\Lambda',
    R', \varphi')(t). 
  \end{displaymath}
  On the other hand, the inclusion \(\iota \colon L' \to L\)
  defines simplicial maps
  \begin{displaymath}
    \iota \colon N\Lambda'_t \to N\Lambda_t
  \end{displaymath}
  and
  \begin{displaymath}
    \iota \colon N(\Lambda', R', \varphi')(t) \to N(\Lambda,
    R, \varphi)(t). 
  \end{displaymath}
  Moreover the above claims imply that the compositions
  \begin{displaymath}
    N\Lambda_t \xto {f_t} N\Lambda'_t \xto \iota N\Lambda_t
  \end{displaymath}
  and
  \begin{displaymath}
    N(\Lambda, R, \varphi)(t) \xto {f_t} N(\Lambda',
    R', \varphi')(t) \xto \iota 
    N(\Lambda, R, \varphi)(t) 
  \end{displaymath}
  are contiguous to identity maps. Since \(f_t \iota\) is the
  identity this implies that geometric realizations of the inclusions
  \begin{displaymath}
    N\Lambda'_t \xto \iota N \Lambda_t
  \end{displaymath}
  and
  \begin{displaymath}
    N(\Lambda', R', \varphi')(t) \xto \iota N(\Lambda, R, \varphi)(t) 
  \end{displaymath}
  are homotopy equivalences. Since we have assumed that
  the geometric realization of the inclusion
  \begin{displaymath}
	N(\Lambda, R, \varphi)(t) \xto \iota N\Lambda_t
  \end{displaymath}
  is not a homotopy equivalence, we can conclude that the geometric
  realization of the inclusion
  \begin{displaymath}
    N(\Lambda', R', \varphi')(t) \xto \iota N\Lambda'_t
  \end{displaymath}
  is not a homotopy equivalence, as desired.
\end{proof}

\medskip


\end{document}